\documentclass{amsart}
\usepackage{amssymb, amsmath}
\usepackage[dvips]{graphicx}
\usepackage{amsfonts}
\usepackage{latexsym}
\usepackage{color}
\newtheorem{theorem}{Theorem}	
\newtheorem{lemma}{Lemma}[section]		
\newtheorem{corollary}{Corollary}		
\newtheorem{proposition}{Proposition}

\begin{document}
\title[Approximation of spherical convex bodies of constant width $\pi/2$]
{Approximation of spherical convex bodies of constant width $\pi/2$}
\author{Huhe Han}
\address{College of Science, Northwest Agriculture and Forestry University, China}
\email{han-huhe@nwafu.edu.cn}
\begin{abstract}
Let $C\subset \mathbb{S}^2$ be a spherical convex body of constant width $\tau$.
It is known that 
(i) if $\tau<\pi/2$ then for any $\varepsilon>0$ 
there exists a spherical convex body $C_\varepsilon$ of constant width $\tau$ 
whose boundary consists only of arcs of circles of radius $\tau$ 
such that 
the Hausdorff distance between $C$ and $C_\varepsilon$ is at most $\varepsilon$;
(ii) if $\tau>\pi/2$ then for any $\varepsilon>0$ 
there exists a spherical convex body $C_\varepsilon$ of constant width $\tau$ 
whose boundary consists only of arcs of circles of radius $\tau-\frac{\pi}{2}$ and great circle arcs
such that 
the Hausdorff distance between $C$ and $C_\varepsilon$ is at most $\varepsilon$.
In this paper, we present an approximation of the remaining case $\tau=\pi/2$, 
that is, if $\tau=\pi/2$ then for any $\varepsilon>0$ 
there exists a spherical polytope $\mathcal{P}_\varepsilon$ of constant width $\pi/2$  
such that 
the Hausdorff distance between $C$ and $\mathcal{P}_\varepsilon$ is at most $\varepsilon$.
\end{abstract}
\maketitle
\section{introduction}
Blaschke presented an approximation theorem as 
for every convex body of constant width $\tau>0$ in the Euclidean
plane $\mathbb{R}^2$ and every  $\varepsilon > 0$ there exists a convex body of constant width 
$\tau$ whose boundary consists only of pieces of circles of radius $\tau$ 
such that the Hausdorff distance between the two
bodies is at most $\varepsilon$ (see \cite{blaschke15, bonnesen-fenchel87}).
A generalization of this fact for normed planes
is given by Lassak (\cite{lassak12}). 
In \cite{Lassak22-1}, Lassak presented an approximation of spherical reduced body of thickness below 
$\pi/2$. As a corollary of this fact, also a spherical version of the theorem of Blaschke is presented as follows:
\begin{theorem}[\cite{Lassak22-1}]\label{theorem1}
For any spherical convex body $C\subset \mathbb{S}^2$ of constant width $\tau<\pi/2$, 
and for any $\varepsilon>0$ 
there exists a body $C_\varepsilon$ of constant width $\tau$ 
whose boundary consists only of arcs of circles of radius $\tau$ 
such that 
\[
h(C, C_\varepsilon)\leq \varepsilon,
\] 
where $h(C_1, C_2)$ means the Hausdorff distance between $C_1$ and $C_2$ .
\end{theorem}
A counterpart of Theorem \ref{theorem1} is given as follows: 
\begin{theorem}[\cite{hanam}]\label{theorem2}
For any spherical convex body $C\subset \mathbb{S}^2$ of constant width $\tau>\pi/2$, 
and for any $\varepsilon>0$ 
there exists a body $C_\varepsilon$ of constant width $\tau$ 
whose boundary consists only of 
arcs of circles of radius $\tau-\frac{\pi}{2}$ and great circle arcs, 
such that 
\[
h(C, C_\varepsilon)\leq\varepsilon,
\] 
where $h$ is the Hausdorff distance.
\end{theorem} 
In the Euclidean plane, a {\it PC curve}, or {\it piecewise circular
curve}, is defined as  a finite sequence of circular arcs or line segments, with the
endpoint of one arc coinciding with the beginning point of the next (\cite{banchoffgiblin94, banchoffgiblin93}). 
Then the boundary of $C_\varepsilon$ from Theorem \ref{theorem2} is a spherical analog of PC curve.
\par
In this paper, we present an approximation of the remaining case $\tau=\pi/2$.
\begin{theorem}\label{theorem3}
For any spherical convex body $C\subset \mathbb{S}^2$ of constant width $\pi/2$, 
and for any $\varepsilon>0$ 
there exists a spherical polytope $\mathcal{P}_\varepsilon$ of constant width $\pi/2$ , 
such that 
\[
h(C, \mathcal{P}_\varepsilon)\leq\varepsilon,
\] 
where $h$ is the Hausdorff distance.
\end{theorem} 
In \cite{hanam}, 
a conjecture was posted as $\lq\lq$any spherical convex body of constant width $\pi/2$ in $\mathbb{S}^n$ can be
approximated by a sequence of spherical convex polytopes of constant width
$\pi/2$.". 
Thus, Theorem \ref{theorem3} asserts that the conjecture is true if $n=2$.
\par 
\bigskip
This paper is organized as follows. In Section 2, preliminaries are given. The
proofs of Theorem \ref{theorem3} is given in Section 3.
\section{preliminaries}
Let $\mathbb{S}^n$ be the unit sphere centered at the origin of $n+1$ dimensional Euclidean space $\mathbb{R}^{n+1}$. 
For any point $P$ of $\mathbb{S}^n$, the hemisphere centered at $P$ is denoted by $H(P)$, that is 
\[
H(P)=\{Q\in \mathbb{S}^n\mid P\cdot Q\geq 0\},
\]
where the dot in the center stands for the standard dot product of two vectors
$P,Q\in \mathbb{R}^{n+1}$. 
For any points $P,Q$ of $\mathbb{S}^n$ and $P\neq -Q$, denote by $PQ$ the spherical great circle arc with end points $P, Q$ such that $-P, -Q$ are not contained in it. 
We say a non empty subset $W$ of $\mathbb{S}^n$ is {\it hemispherical} if there exists a point 
$P\in\mathbb{S}^n$ such that $W\cap H(P)=\emptyset$;
a hemispherical subset $C$ is {\it spherical convex} if $PQ$ is contained in $C$ for any $P,Q\in C$;
and a convex set $C$ is a {\it spherical convex body} if it contains an interior point.
Denote by $\partial C$ the boundary of spherical convex body $C$.
Let $P$ be a point of $\partial C$. 
If 
\[
C\subset H(Q)\ \mbox{and}\ P\in C\cap \partial H(Q),
\] 
then we say the hemisphere $H(Q)$ {\it supports} $C$ at $P$, and $H(Q)$ is a {\it supporting hemisphere} of $C$ at $P$. 
The intersection $H(P)\cap H(Q)$ is called a {\it lune} of $\mathbb{S}^{n}$, where $P\neq Q$. 
The {\it thickness of the lune} $H(P)\cap H(Q)$ is given by 
$
\pi-\arccos (P\cdot Q),
$
and denoted by $\Delta (H(P)\cap H(Q))$. 
If $H(P)$ is a supporting hemisphere of a spherical convex body $C$, 
{\it the width of} $C$ with respect to $H(P)$, denoted by $\mbox{width}_{H(P)}(C)$, is defined by 
\[
\mbox{width}_{H(P)} (C) =\mbox{min}\{\Delta(H(P)\cap H(Q))| H(Q)\ \mbox{supports}\ C\}.
\] We say that a spherical convex body $C$ is {\it of constant width},
if all widths of $C$ 
are equal. 
Following \cite{Lassak15}, 
 {\it the thickness} of a convex body $C\subset \mathbb{S}^n$, denoted by $\Delta(C)$,  is 
the minimum of $\mbox{width}_{H(P)} (C)$ over all supporting hemispheres $H(P)$ of $C$. Namely, 
\[
\Delta(C) = \mbox{min}\{\mbox{width}_{K}(C)| K\ \mbox{is\ a\ supporting\ hemisphere\ of\ }C\}.
\]
We say that a spherical convex body $C$ is of constant diameter $\tau> 0$, if
the diameter of $C$ is $\tau$, and for every point $P$ on the boundary of $C$ there
exists a point $Q$ of $C$ such that $\arccos (P\cdot Q)=\tau$.
\par
For any non-empty subset $W\subset \mathbb{S}^{n}$, the {\it spherical polar set of $W$}, denoted by 
$W^\circ$, is defined as follows: 
\[
W^\circ = \bigcap_{P\in W}H(P).
\]  
It is well known that if $W$ is spherical convex then $W=(W^\circ)^\circ$. 
We say a spherical convex body $C$ is {\it self-dual} if $C=C^\circ$.
\begin{proposition}[\cite{nishimurasakemi2}]\label{propself}
Let $C\subset \mathbb{S}^n$ be a spherical convex body. 
Then, $C=C^\circ$ 
if and only if $C$ is of constant width $\pi/2$.
\end{proposition}
More details on spherical convex body of constant width, see for instance 
\cite{bezdek2023, hanrm, lassak22}.
\begin{lemma}[\cite{hanam}]
Let $C$ be a spherical convex body in $\mathbb{S}^{n}$. The
following two assertions are equivalent:
\begin{enumerate}
\item the hemisphere $H(P)$ supports $C^\circ$ at $Q$;
\item the hemisphere $H(Q)$ supports $C$ at $P$.
\end{enumerate}
\end{lemma}
\section{Proof of Theorem \ref{theorem3}}
\begin{proposition}[\cite{hwam}]\label{diameter-width}
Let $C$ be a spherical convex body in $\mathbb{S}^n$, and $0< \tau < \pi$. 
The following two statements are equivalent:
\begin{enumerate}
\item $C$ is of constant diameter $\tau$ .
\item $C$ is of constant width $\tau$.
\end{enumerate}
\end{proposition}

\begin{lemma}\label{lesspi/2}
Let $P, Q$ be two points of $\mathbb{S}^2$ such that $\arccos(P_1\cdot P_2)=\pi/2$. 
Let $P_3$ be a point of $H(P_1)\cap H(P_2)$. 
Then 
\[
\arccos(Q_1\cdot Q_2)<\pi/2,
\]
where $Q_1, Q_2$ are two points of spherical triangle $P_1,P_2,P_3$ such that $Q_i\neq P_j, i=1,2; j=1,2,3$.
\end{lemma}
\begin{proof}
Without loss of generality, 
set $P_1=(1, 0, 0), P_2=(0,1, 0)$ and $P_3=(0,0,1)$. 
Then we can assume that 
$Q_1=(x_1, y_1, z_1), Q_2=(x_2, y_2, z_2)$ 
where $0\leq x_i, y_i, z_i<1$ such that $x_i^2+y_i^2+z_i^2=1, i=1,2$. 
Since $Q_1\cdot Q_2=x_1x_2+y_1y_2+z_1z_2<1$, 
it follows that $\arccos(Q_1\cdot Q_2)<\pi/2$.
\end{proof}
If spherical convex body $C$ is a spherical polytope, it is nothing to prove. 
Hence, we assume that $C$ is locally strictly convex, namely, the boundary of $C$ always contains a spherical curve different from great circle arc.  
\par
{\bf Part I:} Construct a spherical convex body $C_\varepsilon^1$.
\par
We can parametrize $\partial C$ by a mapping 
$c: [0, 2\pi]\to \mathbb{S}^2$ such that $c(0)=c(2\pi)$ and $c(t_1)\neq c(t_2)$ for any two different numbers 
$t_1, t_2\in (0, 2\pi)$. 
Then, we could say that the curve $\widehat{R_1R_2}$ of $\partial C$ is oriented from $R_1$ to $R_2$. 
Let $T^\prime T^{\prime\prime}$ be a strictly convex part of $\partial C$ with endpoints $T^\prime,T^{\prime\prime}$.
We insert a finite points $P_1, \dots, P_l$ into $T^\prime T^{\prime\prime}$ such that 
$P_iP_{i+1}$ is a great circle arc of $\partial H(R_i)$ and 
\[
R_i\in B_s(C, \varepsilon)=\left\{ P\in \mathbb{S}^2\mid \inf_{Q\in C}\{ P\cdot Q<\varepsilon\}\right\}, 
\] 
for any $1\leq i \leq l-1$. 
Here, $T^\prime=P_1$ and $T^{\prime\prime}=P_l$. 
Next, we gradually construct the convex bodies $C_\varepsilon^i$ of constant width $\pi/2$ and 
remove these strictly convex parts $\widehat{P_iP}_{i+1}$.
Since $C$ is a body of constant width $\pi/2$, by Proposition \ref{diameter-width},  
there exist two points $Q_1, Q_2\in \partial C^\circ$ such that 
\[
\arccos(P_1\cdot Q_1)=\arccos(P_2\cdot Q_2)=\pi/2.
\] 
Since $\varepsilon$ can be a sufficiently small number, we may assume that 
$Q_1, Q_2$ are contained in $H(R_1)$ (see Figure \ref{figure1}).
Then we know that the great circle segment $R_1Q_1$ (resp. $R_1Q_2$) is a subset of the 
boundary of supporting hemisphere
$ H(P_1)$ (resp. $H(P_2)$) of $C$. 
This means that 
$R_1$ is a intersection point of $\partial H(P_1)\cap \partial H(P_2)$. 
We denote by $C_\varepsilon^1$ the spherical convex hull of 
\[
\left(\partial{C}\backslash \widehat{P_1P_2}\right)\cup P_1P_2\cup R_1.
\]
Then by the discussion above, we know that $\partial C_\varepsilon^1$ is the union
\[
\partial{C}\backslash \left(\widehat{P_1P_2}\cup \widehat{Q_1Q_2}\right) \cup P_1P_2\cup R_1Q_1\cup R_1Q_2, 
\] 
where $\widehat{Q_1Q_2}$ (counterclockwise  in the usual sense, see Figure \ref{figure1}) is the curve of $\partial C$ with beginning point $Q_1$ and endpoint $Q_2$.

\begin{figure}[htbp]\label{figure1}
  \includegraphics[width=8cm]{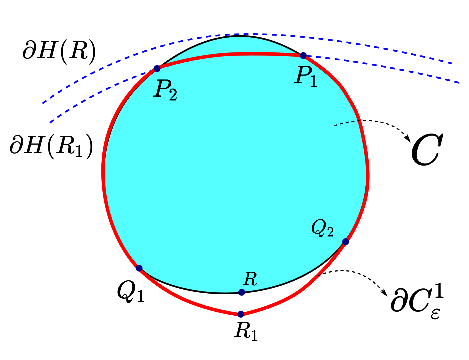}
  \caption{A spherical convex body $C_\varepsilon^1.$}
  \label{figure1}
\end{figure}
\par
By the construction, it is clear that the length of strictly convex part of 
$\partial C_\varepsilon^1$ is smaller than 
the length of strictly convex part of $\partial C$.
\par
\bigskip
{\bf Part II:} The spherical convex body $C_\varepsilon^1$ is of constant width $\pi/2$. 
\par
By Proposition \ref{diameter-width}, it is sufficient to prove $C_\varepsilon^1$ is of constant diameter $\pi/2$.
\begin{lemma}\label{lemmaintersect}
The chords $P_1Q_1, P_2Q_2$ of $C$ intersect each other.
\end{lemma}
\begin{proof}
If $P_1Q_1, P_2Q_2$ disjoint each other, then $P_2,Q_2$ are two points of a spherical triangle 
$P_1Q_1R^\prime$, where $R^\prime\in \partial H(P_1)\cap \partial H(Q_1)$. 
By Lemma \ref{lesspi/2}, $\arccos(P_2\cdot Q_2)<\pi/2$.
This constradicts the fact that $\arccos(P_2\cdot Q_2)=\pi/2$. 
\end{proof} 
By the proof of Lemma \ref{lemmaintersect}, one obtains that any two chords of $C$ 
that are equal to the diameter intersect each other. 
Let $P$ be a relative interior point of $\widehat{P_2Q_1}$. 
Since $\arccos(P_1\cdot Q_1)=\pi/2$ and $C$ is a spherical convex body of constant diameter,
by Lemma \ref{lemmaintersect}, 
there exists a point $P^\prime\in \widehat{Q_1P_1}$ such that $\arccos(P\cdot P^\prime)=\pi/2$.
On the other hand, since $\arccos(P_2\cdot Q_2)=\pi/2$, applying Lemma \ref{lemmaintersect} again, 
we know that $PP^\prime$ and $P_2Q_2$ intersect each other.
Thus, we know that $P^\prime$ is a point of $\widehat{Q_2P_1}$.
Similarly, for any point $Q\in \widehat{Q_2P_1}$ there exists a point $Q^\prime\in \widehat{P_2Q_1}$ 
such that $\arccos(Q\cdot Q^\prime)=\pi/2$. 
From above discussion we know that $C_\varepsilon^1$ is a body of constant diameter $\pi/2$, that is for any point
$R\in \partial C_\varepsilon^1$ there exists a point $R^\prime \in \partial C_\varepsilon^1$ such that
$\arccos(R\cdot R^\prime)=\pi/2$.
\par
\bigskip
{\bf Part III:} $h(C,C_\varepsilon^1)<\varepsilon/2$.
\par
By the construction of Part I, it is clear that $C_\varepsilon^1$ is a subset of $B_s(C, \varepsilon)$.
\par
On the other hand, by the assumption $R_1\in B_s(C, \varepsilon)$, 
there exists a point $R\in \partial C$ such that $\arccos(R\cdot R_1)<\varepsilon$. 
This implies that the spherical convex hull of $P_1P_2\cup \widehat{P_1P_2}$ is a subset of the lune
$H(R)\cap H(R_1)$. 
Notice that the thickness of lune  $H(R)\cap H(R_1)$ is smaller than $\varepsilon$, that is 
\[
\Delta(H(R)\cap H(R_1))<\varepsilon.
\]
Then we know that for any point $P$ of $P_1P_2$ there exist a point 
$\widehat{P}\in \widehat{P_1P_2}$ such that $\arccos(P\cdot\widehat{P})<\varepsilon$ (see Figure 2).
This implies  $C_\varepsilon$ is a subset of $B_s(C_\varepsilon^1, \varepsilon)$.
Thus we conclude $h(C,C_\varepsilon^1)<\varepsilon$.
\begin{figure}[htbp]\label{figure3}
  \includegraphics[width=8cm]{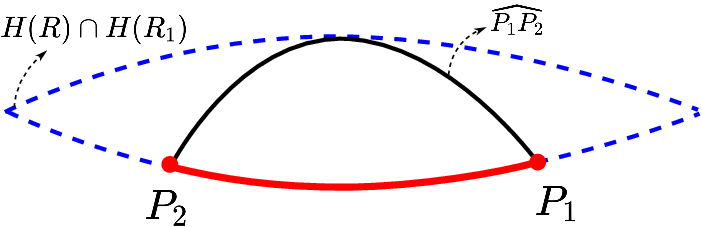}
 \caption{$\Delta(H(R)\cap H(R_1))<\varepsilon$.}
\end{figure}
\par
\bigskip
We repeat the steps in Part II.
And then, after a finite number steps we can remove the strictly convex part $T^\prime T^{\prime\prime}$ and obtain a spherical polytope $C_{\varepsilon, 1}$ of constant width $\pi/2$ such that 
\[
h(C, C_{\varepsilon, 1})<\varepsilon.
\]
If $C_{\varepsilon,1}$ is a spherical polytope, then put $\mathcal{P}_\varepsilon=C_{\varepsilon,1}$. 
In the opposite case, we replace $C_{\varepsilon,1}$ (resp. $\varepsilon$) with $C$ (resp. $\varepsilon/2$) and repeat the step in Part I. 
And then, after a finite number of $m$ steps we get a spherical polytope $C_{\varepsilon, m}$ of constant width $\pi/2$, and 
\[
h(C, C_{\varepsilon,m})<h(C, C_{\varepsilon,1})+h(C, C_{\varepsilon,2})+\dots+h(C, C_{\varepsilon,m})<\varepsilon+\frac{\varepsilon}{2}+\dots+\frac{\varepsilon}{2^m}<2\varepsilon.
\]
Put $\mathcal{P}_\varepsilon=C_{\varepsilon,m}$. By the arbitrariness of $\varepsilon$ , this completes the proof.
\bigskip
\par
{\bf Acknowledgements.}
This work was supported, in partial, by
Natural Science Basic Research Program of Shaanxi (Program No.2023-JC-YB-070).


\begin{thebibliography}{99}
\normalsize
\baselineskip=17pt
\bibitem{banchoffgiblin94}T. Banchoff and P. Giblin, 
\emph{On the geometry of piecewise circular curves},  Amer. Math. Monthly, {\bf 101} (1994), 403--416.
\bibitem{bezdek2023}K. 
Bezdek, 
\emph{On a strengthening of the
Blaschke-Leichtweiss theorem}, J. Geom. (2023) 114:2.
\bibitem{blaschke15}
W. Blaschke, \textit{Konvexe Bereiche gegebener konstanter Breite und kleinsten Inhalts}, Math.
Ann. {\bf 76} (1915) 504--513.
\bibitem{bonnesen-fenchel87}
T. Bonnesen and W. Fenchel, 
\textit{Theorie der konvexen K\"{o}rper}, Springer, Berlin et al., 1934, (Engl.
transl. \textit{Theory of Convex Bodies}, BCS Associated, Moscow, Idaho USA, 1987).
\bibitem{banchoffgiblin93}P.J. Giblin and T. F. Banchoff, 
\emph{Symmetry sets of piecewise-circular curves}, Proc. Royal Soc. Edinburgh, {\bf 123A} (1993), 1135 --1149.
\bibitem{hanam}H. Han, 
\textit{Behavior of convex integrand at a d-apex of its Wulff shape and approximation of spherical bodies of constant width}, Aequationes Math. (2024).
\bibitem{hanrm}H. Han, \emph{Self-dual polytope and self-dual smooth Wulff shape},
Results Math. {\bf 79} 134 (2024). 
\bibitem{hnjmsj}H.~Han and T.~Nishimura,
\emph{Self-dual Wulff shapes and spherical convex bodies of constant
	width {$\pi$}/2}, J. Math. Soc. Japan., {\bf 69}, (2017) 1475--1484.
\bibitem{hwam}H.~ Han and D. ~Wu, \emph{Constant diameter and constant width of spherical convex bodies}, Aequationes Math., {\bf 95} (2021), 167--174.
\bibitem{lassak12}M. Lassak, \textit{Approximation of bodies of constant width and reduced bodies in a normed plane}, J. Convex Anal. {\bf 19} (2012) 865--874.
\bibitem{Lassak15}M.~Lassak,
\emph{Width of spherical convex bodies},
Aequationes Math., {\bf 89} (2015), 555--567.
\bibitem{lassak22}M.~Lassak, \emph{Spherical Geometry-A Survey on Width and Thickness of Convex Bodies}, In: Papadopoulos, A. (eds) Surveys in Geometry I. Springer, Cham, (2022)
\bibitem{Lassak22-1}M.~Lassak,
\emph{Approximation of spherical convex bodies of constant width and reduced bodies}, J. Convex Analysis, {\bf 29} (2022), 921--928. 
\bibitem{LM18}M.~Lassak and M.~Musielak,
\emph{Spherical bodies of constant width},
Aequationes Math., {\bf 92} (2018), 627--640.
\bibitem{nishimurasakemi2} T.\ Nishimura and Y.\ Sakemi,
\emph{Topological aspect of Wulff shapes},
J. Math. Soc. Japan, {\bf 66} (2014), 89--109.
\bibitem{martini2019}
H. Martini, L. Montejano, D. Oliveros, 
\emph{Bodies of Constant Width. An Introduction to Convex Geometry with Applications} (Springer Nature Switzerland AG, 2019)



\end{thebibliography}
\end{document}